\documentclass{amsart}[12pt]

\usepackage{amsthm, bbold, amssymb, enumerate}
\DeclareMathOperator\Gal{Gal}

\usepackage[usenames, dvipsnames]{color}
\definecolor{mygray}{gray}{0.6}

\newcommand{\authnote}[2][]{\noindent {\if!#1!  {\bf TODO} \else {\small \bf #1} \fi: #2} \vspace{0.1in}}

\def\1{\mathbb{1}}

\def\cE{\mathcal{E}}

\def\cO{\mathcal{O}}

\def\cZ{\mathcal{Z}}
\def\cK{\mathcal{K}}

\def\A{\mathbb{A}}

\def\Z{\mathbb{Z}}

\def\Q{\mathbb{Q}}
\def\R{\mathbb{R}}
\def\C{\mathbb{C}}

\def\T{\mathbf{T}}

\def\U{\mathbf{U}}
\def\G{\mathbf{G}}
\def\H{\mathbf{H}}

\def\Res{\text{Res}}

\def\ab{\text{ab}}

\def\Gal{\text{Gal}}

\usepackage[centertags]{amsmath}
\usepackage{amsfonts}
\usepackage{amssymb}
\usepackage{amsthm}
\usepackage{amscd}
\usepackage{hyperref}
\usepackage{xypic}

\DeclareFontFamily{OT1}{rsfs}{}
\DeclareFontShape{OT1}{rsfs}{n}{it}{<-> rsfs10}{}
\DeclareMathAlphabet{\mathscr}{OT1}{rsfs}{n}{it}

\newcommand{\cN}{\mathcal N}

\newcommand{\cA}{\mathcal{A}}

\newcommand{\cI}{\mathcal{I}}

\newcommand{\cV}{\mathcal{V}}

\newcommand{\ds}{\displaystyle}

\newcommand{\cB}{\mathcal{B}}

\newcommand{\ra}{\rightarrow}

\newcommand{\mthree}[9]{\left [
        \begin{matrix}#1&#2&#3\\#4&#5&#6\\#7&#8&#9
        \end{matrix}\right ]}

\newcommand{\comment}[1]{}

\newcommand{\cS}{\mathcal{S}}

\newcommand{\cU}{\mathcal{U}}

\DeclareFontEncoding{OT2}{}{} 

\renewcommand{\H}{\mathbf{H}}

\DeclareMathOperator{\dist}{dist}

\DeclareMathOperator{\Tr}{Tr}

\DeclareMathOperator{\GL}{\mathbf{GL}}

\DeclareMathOperator{\Hbf}{\mathbf{H}}

\DeclareMathOperator{\Sh}{Sh}
\DeclareMathOperator{\End}{End}

\DeclareMathOperator{\inv}{\mathbf{inv}}

\DeclareMathOperator{\pr}{pr}

\DeclareMathOperator{\Art}{Art}

\DeclareMathOperator{\Hyp}{\mathbf{Hyp}}

\DeclareMathOperator{\Frac}{Frac}

\theoremstyle{plain} 
\newtheorem{thm}{Theorem}[section] 
\newtheorem{prop}[thm]{Proposition}
\newtheorem{cor}[thm]{Corollary}
\newtheorem{lem}[thm]{Lemma}

\theoremstyle{definition} 
\newtheorem{assumptions}[thm]{Assumptions}
\theoremstyle{remark} 
\newtheorem{rem}{Remark}

\newcounter{tasknumber}
\newcommand{\task}[2][]{%
  \addtocounter{tasknumber}{1}%
  \begin{center}%
  \framebox[1.1\width]{\begin{minipage}{0.9\textwidth}%
  \textbf{Task \arabic{tasknumber}} \textit{\if!#1(unassigned)!\else (#1)\fi}: {#2}%
  \end{minipage}}%
  \end{center}%
}

\newcounter{assumptionnumber}
\newcommand{\assumption}[2][]{%
  \addtocounter{assumptionnumber}{1}%
  \begin{center}%
  \framebox[1.1\width]{\begin{minipage}{0.9\textwidth}%
  \textbf{Assumption \arabic{assumptionnumber}} \textit{\if!#1!\else (#1)\fi}: {#2}%
  \end{minipage}}%
  \end{center}%
}

\usepackage{color}
\usepackage{lastpage}
\usepackage{centernot}
\usepackage{tikz-cd}
\usepackage{tikz}
\usepackage{graphicx}

\title[Vertical Distribution Relations for Unitary Shimura Varieties]{Vertical Distribution Relations for Special Cycles on Unitary Shimura Varieties}

\author{R\'eda Boumasmoud}
\email{reda.boumasmoud@epfl.ch}
\address{Ecole Polytechnique F\'ed\'erale de Lausanne, Switzerland}

\author{Ernest Hunter Brooks}
\email{ernest.brooks@epfl.ch}
\address{Ecole Polytechnique F\'ed\'erale de Lausanne, Switzerland}

\author{Dimitar Jetchev}
\email{dimitar.jetchev@epfl.ch}
\address{Ecole Polytechnique F\'ed\'erale de Lausanne, Switzerland}

\begin{document}
\maketitle

\begin{abstract}
We consider cycles on a 3-dimensional Shimura varieties attached to a unitary group, defined over extensions of a CM field $E$, which appear in the context of the conjectures of Gan, Gross, and Prasad \cite{gan-gross-prasad}. We establish a vertical distribution relation for these cycles over an anticyclotomic extension of $E$, complementing the horizontal distribution relation of \cite{jetchev:unitary}, and use this to define a family of norm-compatible cycles over these fields, thus obtaining a universal norm construction similar to the Heegner $\Lambda$-module constructed from Heegner points.
\end{abstract}

\section{Introduction}\label{sec:intro}
Let $\cK/\Q$ be an imaginary quadratic field with ring of integers $\cO_K$ and let $\cN$ be an ideal of $\cO_{\cK}$ of norm $N$. If $m$ is prime to $N$, the isogeny $\C/\cO_{m} \to \C/(\cN \cap \cO_{m})^{-1}$ corresponds to a Heegner point $x_{m}$ in $X_0(N)(\cK[m])$, where $\cK[m]$ denotes the ring class field of conductor $m$ and $\cO_m = \Z + m \cO_\cK$ is the corresponding order of $\cK$.

Let $\cE$ be an elliptic curve over $\Q$ of conductor $N$. For applications to anticyclotomic Iwasawa theory, one would like a module of universal norms in $\cE(\cK[p^n])$ as $n$ varies; that is, a collection of Heegner points $y_{p^n} \in E(\cK[p^n])$ such that
\begin{equation}\label{compatible}
\Tr_{\cK[p^{n+1}]/\cK[p^n]} y_{p^{n+1}} = y_{p^n}.
\end{equation}
The images $\widetilde{y}_{p^n}$ of the points $x_{p^n}$ constructed above under a fixed modular parametrization $X_0(N) \to \cE$ do not satisfy this relation, but instead satisfy the ``vertical distribution relation'' (see~\cite[Lem.2]{perrin-riou:iwasawa}): 
$$
\Tr_{\cK[p^{n+1}]/\cK[p^n]} \widetilde{y}_{p^{n+1}} = a_p \widetilde{y}_{p^n} - \widetilde{y}_{p^{n-1}}, \qquad n > 1. 
$$
As explained in \cite[p.3]{nekovar:parity2}, this relation, together with standard techniques from the theory of linear recurrences, allows one to modify the cycles $\widetilde{y}_{p^n}$ into a family satisfying (\ref{compatible}).

This article establishes, under the assumption that $p$ is inert in the CM field, a vertical distribution relation for some higher-dimensional Shimura varieties, where the embedding of the non-split torus $\Res_{\cK/\Q} \G_m \to \GL_2$ defining Heegner points is replaced by an embedding of unitary groups defining special one-dimensional cycles on a Shimura threefold. These cycles have their origin in the conjectures of Gan, Gross and Prasad \cite{gan-gross-prasad}; the intersection theory of variants of these cycles has been studied in the work of Howard \cite{howard:kr}, and work on a Gross--Zagier formula for them has been initiated via the arithmetic fundamental lemma of Zhang \cite{zhang:afl} and Rapoport, Terstiege and Zhang \cite{rapoport-terstiege-zhang}. We work with the versions of these cycles defined in \cite{jetchev:unitary}, where a horizontal distribution relation is proven (again under the assumption that $p$ is inert).

\subsection{Main theorem}\label{subsec:thm}
Let $F$ be a totally real field and let $E/F$ be a totally imaginary quadratic extension, for which we pick once and for all an embedding into $\mathbb{C}$. Let $\tau$ be a finite prime of $F$, \emph{inert} in $E$, and fix an embedding of $\overline{F}$ into $\overline{F}_\tau$; we will continue to write $\tau$ for the prime in any finite extension of $F$ by this choice.  Let $W \subset V$ be an embedding of $E$-hermitian spaces with signatures $(1,1)$ (resp. $(2,1)$) at the distinguished real place of $F$ and $(2, 0)$ (resp. $(3,0)$) at the other real places. One has algebraic groups $\G = \Res_{E/F}(\U(V) \times \U(W))$ and $\H = \Res_{E/F} \U(W)$, and an embedding $\H \hookrightarrow \G$, described in Section \ref{sec:notation}. In Section \ref{sec:sv}, a particular compact $K \subset \G(\A_f)$ (for which $\tau$ is \emph{allowable} in the sense of \cite[Defn.1.1]{jetchev:unitary}, as recalled in Section \ref{subsec:allowable}), and Hermitian symmetric domain $X$ are chosen, which give rise to a Shimura variety $\Sh_K(G, X)$ and a family $\cZ_K(\G, \H)$ of special one-cycles on this threefold. The cycles in $\cZ_K(\G, \H)$ are defined over abelian extensions of $E$.

Attached to this data is the Hecke polynomial given by $\ds H_\tau(z) = \sum_{i=0}^6 C_i z^i\in \mathcal{H}_\tau[z]$, where $\mathcal{H}_\tau$ is a local Hecke algebra whose definition is recalled in Section~\ref{sec:hecke}.

Write $E[\tau^n]$ for the ring class field of $E$ of conductor $\tau^n$, that is, the abelian extension of $E$ whose norm subgroup is $E^\times \cdot \widehat{\cO_{\tau^n}}^\times \subset \widehat{E}^\times$ where $\cO_{\tau^n} = \cO_F + \tau^n \cO_{E}$ (here, $\tau^n$ denotes the $n$th power of the prime ideal of $\cO_F$ corresponding to $\tau$). If $L$ is any extension of $E$, write $L[\tau^n]$ for the compositum $L \cdot E[\tau^n]$. Our main theorem holds under two assumptions:

\begin{assumptions} \label{assumptions}

\begin{enumerate}[A.]
\item There exists a cycle $\xi_1 \in \cZ_K(\G, \H)$ defined over a finite extension $L$ of the Hilbert class field $E[1]$, which is abelian over $E$ and in which the chosen extension of $\tau$ to $E[1]$ splits completely.
\item The local invariants at $\tau$ (see Section \ref{subsec:galois} or \cite[Prop.3.4]{jetchev:unitary} for the definition) are given by $\inv_\tau(\xi_1) = (0,0)$.

\end{enumerate}
\end{assumptions}
Fixing such an $L$, the vertical distribution relation is then:

\begin{thm}\label{thm:main}
Suppose that $\tau$ is allowable for $(\G, \H, K)$ as defined in Paragraph \ref{subsec:allowable}. Under the assumptions listed in \ref{assumptions}, there exists a family of cycles $\xi_n \in \cZ_K(\G, \H)$ such that
\begin{itemize}
\item The field of definition of $\xi_n$ is $L[\tau^n]$ .
\item For all sufficiently large $n$, one has
$$
\Tr_{L[\tau^{n+6}]/L[\tau^{n+5}]} \left( C_6 \xi_{n+6} + \dots + C_1\xi_{n+1} + C_0 \xi_n \right) = 0.
$$
\end{itemize}
\end{thm}

\begin{rem}
It would be useful to have arithmetic conditions to guarantee Assumption \ref{assumptions}A in terms of a ``Heegner hypothesis'' on the pair $(E, K)$, particularly in the case that $L = E[1]$ is the Hilbert class field of $E$. This would require an extension of the results of \cite{jetchev:unitary} to split and ramified primes. The case of general $L$ may be necessary for arithmetic applications (c.f. \cite{aflalo-nekovar} for an instance where such a generalization is needed for $\GL_2$) and the result is no harder to prove. Assumption (\ref{assumptions})B holds for almost all allowable $\tau$, provided that Assumption(\ref{assumptions})A holds. 
\end{rem}

As explained in Section \ref{sec:applications}, the above theorem can be used, together with a suitable choice of representation of the local group $\U(V)(F_\tau) \times \U(W)(F_\tau)$, to construct norm-compatible families $\{\widetilde{\xi}_n\}$. 

There is a variant of this theorem using fewer terms, which may be more useful for computation -- see Remark \ref{shorter_relation}.

\subsection{Strategy of proof}
Theorem \ref{thm:main} would follow formally from the following ``facts,'' if they were true:
\begin{itemize}
\item There are operators $\cU$ and $\cV$ on $\cZ_K(\G, \H)$ such that $\xi_n:= \cU^n \xi_0$ is defined over $E[\tau^n]$ and, for sufficiently large $n$, $\cV \xi_n = \xi_{n-1}$.
\item This operator $\cV$ is a formal root of the Hecke polynomial, in the sense that $H_\tau(\cV)$ induces the $0$ endomorphism of $\mathbb{Z}[\cZ_K(\G, \H)]$.
\end{itemize}
Formalizing these ``facts'' is difficult on the level of the cycles themselves, but turns out to work on the level of the Bruhat--Tits building attached to ($\G$, $\tau$), which is a product of two trees. We recall the definitions of the cycles and buildings in Section \ref{sec:sv}, and then introduce the $\cV$ operator in Section \ref{sec:hecke}, showing that it is a root of the Hecke polynomial; we then show how to descend this to the level of cycles in Section \ref{sec:proof}. We conclude with Section \ref{sec:applications}, which explains how to build norm-compatible families.

%
%

\section{Unitary Shimura Varieties}\label{sec:sv}
In this section we recall the constructions and notation of \cite[Section 2]{jetchev:unitary}. The reader is referred to \emph{loc. cit.} for proofs and references for these facts.

\subsection{Global fields and unitary groups}\label{sec:notation}
Write $D$ for the orthogonal complement of $W$ in $V$. Let $\G_V = \Res_{F/\Q} \U(V)$ and $\G_W = \Res_{F/\Q} \U(W)$; these are algebraic groups over $\Q$. The decomposition $V = W \perp D$ gives an embedding $\G_W \hookrightarrow \G_V$: if $R$ is a $\Q$-algebra, then $\U(W)(R)$ acts on $V \otimes R = (W \otimes R) \oplus (D \otimes R)$ by acting trivially on $D \otimes R$. Let $\G = \G_V \times \G_W$; there is a diagonal embedding
$$
\Delta \colon \G_W \hookrightarrow \G.
$$
Let $\H$ be the subgroup $\Delta(\G_W)$ of $\G$.

Recall that we have fixed a prime $\tau$ of $F$ that is \emph{inert} in $E$ and let $p$ be the rational prime below $\tau$. Let $q$ be the residue cardinality of $\tau$ (as a place of $F$), and choose a uniformizer $\varpi$ for $F_\tau$ (hence also for $E_\tau$).

Write $G_{V, \tau} = \U(V)(F_\tau)$ and $G_{W, \tau} = \U(W)(F_\tau)$. Let $G_\tau$ be the product $G_{V, \tau} \times G_{W, \tau}$ and write $H_\tau$ for the diagonally embedded copy of $G_{W, \tau}$ in $G_\tau$ (note that the symbol $H_\tau$ can denote either this group or the Hecke polynomial, depending on the context).

\subsection{Hermitian symmetric domains}
Let $X_V$ be the set of negative definite lines in $V \otimes \R$ (the tensor product taken with respect to the distinguished embedding), and similarly $X_W$ the set of negative definite lines in $W \otimes \R$. 

Setting $X = X_V \times X_W$, the diagonal embedding $W \hookrightarrow V \oplus W$ induces an embedding of $X_W$ into $X$; write $Y$ for the image of $X_W$.

\subsection{Compact-open subgroups of $\G(\A_f)$} \label{subsec:allowable}
Fix a compact-open subgroup $K$ of $\G(\A_f)$. We now make the assumption that $\tau$ is allowable for $(\G, \H, K)$ in the sense of \cite[Defn.1.1]{jetchev:unitary}; namely, writing $K_\tau$ for $K \cap \G_\tau \subset \G(\A_f)$, we assume that
\begin{itemize}
\item The groups $\U(V)_{F_\tau}$ and $\U(W)_{F_\tau}$ are quasi-split.
\item One has $K_\tau = K_{V, \tau} \times K_{W, \tau}$, where $K_{V, \tau} $ and $K_{W, \tau}$ are hyperspecial maximal compact subgroups of $G_{V,\tau}$ and $G_{W, \tau}$, respectively, and $K_{V, \tau} \cap G_{W, \tau} = K_{W, \tau}$ (the intersection taken under the given embedding).
\item One has $K_{V, \tau} = K_\tau \times K_V^{(\tau)}$, where $K^{(\tau)} = K_V \cap \U(V)(\A_{F, f}^{(\tau)})$ (where $\A_{F, f}^{(\tau)}$ denotes the finite id\`eles outside of $\tau$) 
\end{itemize}

This assumption implies that there is a Witt basis $\{e_+, e_0, e_- \}$ of the hermitian space $V_\tau := V \otimes E_\tau$, i.e. a basis with respect to which the pairing is given by the matrix $\begin{pmatrix} 0 & 0 & 1 \\ 0 & 1 & 0 \\ 1 & 0 & 0 \end{pmatrix}$. Moreover, this basis has the properties that
\begin{itemize}
\item  $W_\tau := W \otimes E_\tau$ is spanned by $\{e_+, e_-\}$.
\item $K_{V, \tau}$ is the stabilizer in $G_{V, \tau}$ of the $\cO_{E, \tau}$-lattice $\Lambda_V$ generated by $\{e_+, e_0, e_-\}$.
\item $K_{W, \tau}$ is the stabilizer in $G_{W, \tau}$ of the $\cO_{E, \tau}$-lattice $\Lambda_W$ generated by $\{e_+, e_-\}$.
\end{itemize}
Note that the lattices $\Lambda_V$ and $\Lambda_W$ are self-dual.

\subsection{Complex Shimura varieties and special cycles}
The data $(\G, X)$ and $(\H, Y)$ satisfy Deligne's axioms for Shimura data; one computes (see e.g. \cite[\textsection 2.2.6]{jetchev:unitary})  that the reflex field is $E$ in both cases, and thus there are varieties defined over $E$ whose complex points are given by 
\[\Sh_{K_V} (\G_V , X_V )(\C)=\G_V(\Q)\backslash (\G_V(\A_f)\times X_V)/K_V,\]
\[\Sh_{\Delta(K_W)} (\H , Y)(\C)=\H(\Q)\backslash (\H(\A_f)\times Y)/\Delta(K_W).\]
One also has $\Sh_K(\G,X)=\Sh_{K_V} (\G_V , X_V ) \times \Sh_{K_W} (\G_W , X_W )$, with complex points given by 
\[\Sh_{K} (\G , X )(\C)=\G(\Q)\backslash (\G(\A_f)\times X)/K.\]
For any $g \in \G(\A_f)$, there is a ``special cycle'' $\mathcal{Z}_K(g)$, which is the image of $gK\times Y$ in $\Sh_K(\G,X)(\C)$; it is a subvariety of $\Sh_K(\G, X)_\C$. Let $\cZ_K(\G, \H)$ denote the set of all cycles obtained in this way. It is shown in \cite[\textsection 2.3]{jetchev:unitary} that the association $g \mapsto \mathcal{Z}_K(g)$ induces a bijection
\[\mathcal{Z}_K(\G,\H)\simeq  \text{N}_{\G(\Q)}(\H(\Q))\backslash \G(\A_f)/K,\]
where the normalizer is explicitly given by
$$ \text{N}_{\G(\Q)}(\H(\Q))= \Delta(\H(\Q))\left(\underbrace{1\times \U(W^\perp)(\Q)}_{\subset{\G_V(\Q)}}\times \underbrace{\mathbf{Z}_\H(\Q)}_{\subset{\G_W}(\Q)}\right)\subset\G(\Q).$$
 
\subsection{Galois action on cycles}\label{subsec:orbits}
It is shown in \cite[\textsection 2.3]{jetchev:unitary}, using Shimura reciprocity, that the cycles $\mathcal{Z}_K(g)$ are defined over abelian extensions of $E$. Explicitly, given $\sigma \in \Gal(E^{\ab} / E)$, let $s_\sigma \in \mathbb{A}_{E}^\times$ be any element such that $\text{Art}_E(s_\sigma) = \sigma$ where 
$$
\Art_E \colon E^\times \backslash  \mathbb{A}_{E}^\times \ra \Gal(E^\ab / E)
$$
is the Artin map. Let $\T^1 = \U^1_\Q$ and let $\nu \colon \H \to \T^1$ be the determinant map. Consider the homomorphism $r = (r_f, r_\infty) \colon \A_E^\times \to \T^1(\A)$ defined by $r(s)={\overline{s}}/{s}$. Then there exists $h_\sigma \in \H(\A_f)$ such that $\nu(h_\sigma) = r_f(s)$, and for any such choice, one has
\begin{equation}\label{eq:shimurarec}
\sigma(\mathcal{Z}_K(g))=\mathcal{Z}_K(h_\sigma g).
\end{equation}

This description implies that the Galois orbits of cycles receive a surjection
\[ 
\H(\A_f) \backslash \G(\A_f) / K \twoheadrightarrow  \Gal(E^{\ab}/E) \backslash \mathcal{Z}_K(\G, \H).  \]
The allowability hypothesis (Section \ref{subsec:allowable}) implies that domain of this map is of the form
\[
H_\tau \backslash G_\tau \slash K_\tau \times  \H(\A_f^{(\tau)})\backslash \G(\A_f^{(\tau)})\slash K^{(\tau)}.
\]

\subsection{Buildings for unitary groups}\label{subsec:buildings}

The local factor $H_\tau \backslash G_\tau \slash K_\tau$ appearing in the domain of the map  above can be described in terms of the 
Bruhat--Tits building for $G_\tau$. This building is a product of two buildings; one for 
$G_{V, \tau}$ and the other for $G_{W, \tau}$. Each of these buildings is, in turn, isomorphic to a bicolored graph which we now describe.  The reader is referred to \cite[\textsection 4.1]{koskivirta:congruence}  for proofs of the facts below and more details on the buildings, and to \cite[Figure 1]{jetchev:unitary} for a picture.

A ``hyperspecial lattice'' is a lattice $L$ of $V_\tau$ which is self-dual, and a ``special lattice'' $L$ is a lattice which is almost self-dual, which means that one has strict containments  $\varpi L^\vee \subsetneq L \subsetneq L^\vee$. The (underlying bicolored graph of the) Bruhat--Tits building for $G_{V, \tau}$, which we denote by $\mathcal{B}(V_\tau)$, consists of a black vertex for each hyperspecial lattice, and a white vertex for each special lattice. Two vertices are connected by an edge if and only if the corresponding lattices have index $q$ in one another. One calculates that each black vertex has $q^3+1$ white neighbors and each white vertex has $q + 1$ black neighbors.  A choice $\{ f_+, f_0, f_- \}$ of Witt basis for $V_\tau$ determines an apartment in this building whose hyperspecial vertices are the self-dual lattices $\langle \pi^m f_+, f_0, \pi^{-m} f_-\rangle$ for $m \in \mathbb{Z}$; a ``half-apartment'' is the subset of an apartment where $m \geq n$ for some fixed $n$.

One defines $\mathcal{B}(W_\tau)$ similarly; in this case, each black vertex has $q+1$ white neighbors, and each white vertex has $q+1$ black neighbors. 

The building $\mathcal{B}(G_\tau)$ is then the product of these graphs. The group $G_\tau$ acts on $\mathcal{B}(G_\tau)$, preserving incidence relations and geodesics. As this action is transitive on the set of pairs of hyperspecial lattices,  the quotient $G_\tau / K_\tau$ is identified with the set of (pairs of) black vertices in $\mathcal{B}(G_\tau)$. The black vertices of $\mathcal{B}(V_\tau)$, resp. $\mathcal{B}(W_\tau)$, resp. the pairs of black vertices in $\mathcal{B}(G_\tau)$ will be described in the sequel as $\Hyp_{V, \tau}$, resp. $\Hyp_{W, \tau}$, resp. $\Hyp_\tau$. The sets $\Hyp_{V, \tau}$ and $\Hyp_{W, \tau}$ are endowed with distance functions, normalized so that the distance between two neighboring black points (i.e. two black points that share a white neighbor in the bicolored graph) is $1$.

Note that the choice of lattices in Section \ref{subsec:allowable} distinguishes a particular black vertex in each graph; we will informally refer to this vertex as the ``origin,'' and to their product as the ``origin'' in the product building.
 
\subsection{Galois action via Bruhat--Tits buildings}\label{subsec:galois}
 
One can use the building to compute the orbits of the Galois action on the cycles. Given a point $x=(L_{V_\tau},L_{W_\tau}) \in \mathbf{Hyp}_\tau$, write  $$\inv_\tau(x):=\left(\text{dist}(L_{V_\tau},\pr_{W_\tau}(L_{V_\tau})),\text{dist}(L_{W_\tau},\pr_{W_\tau}(L_{V_\tau}))\right),$$ where $\pr_{W_\tau}$ is the projection as in \cite[\textsection 3]{jetchev:unitary}. The following result classifies the 
$H_\tau$-orbits in $\mathbf{Hyp}_\tau$ \cite[Prop.3.4]{jetchev:unitary} 
\begin{prop}\label{prop:orbits}
Two hyperspecial points $x,y\in \mathbf{Hyp}_\tau$ lie in the same $H_\tau$-orbit if and only if $\inv_\tau(x)=\inv_\tau(y)$.
\end{prop}

 Write $\mathbb{U}^1(n):=\nu_\tau(\mathcal{O}_n^\times)$, where $\nu_\tau \colon E_\tau^\times \to \mathbb{U}^1(L_\tau)$ is given by $\nu_\tau(s)=\overline{s}/s$, and $\mathcal{O}_n=\mathcal{O}_{E, \tau}+ \varpi^n \mathcal{O}_{F, \tau}$. The Shimura reciprocity law implies that
 $$\det{\text{Stab}_{H_\tau}(L_{V_\tau},L_{W_\tau})}=\mathbb{U}^1(c_\tau([L_{V_\tau},L_{W_\tau}])),$$ where $c_\tau$ denotes the local conductor given by \cite[Thm.1.1]{jetchev:unitary}, i.e.
 $$c_\tau([L_{V_\tau},L_{W_\tau}]):=q^{\min\{\text{dist}(L_{V_\tau},\pr_{W_\tau}(L_{V_\tau})),2\text{dist}(L_{W_\tau},\
 \pr_{W_\tau}(L_{V_\tau}))\}}.$$

%
%

\section{Hecke operators and partial Hecke operators}\label{sec:hecke}
\subsection{The Hecke polynomial}
The local Hecke algebra $\mathcal{H}_\tau = \mathcal{H}(G_\tau, K_\tau)$ is the set of $K_\tau$-bi-invariant continuous compactly-supported $\mathbb{Z}$-valued functions on $G_\tau$. There are natural actions of $\mathcal{H}_\tau$ on $\mathcal{B}(G_\tau)$ and on the space $\cZ_{K}(\G, \H)$, compatible with the map defined at the end of Section \ref{subsec:orbits}. Explicitly, given an element $g\in G_\tau$, let $\mathbf{1}_{K_\tau g K_\tau}$ be the characteristic function of the double coset $K_\tau g K_\tau$ (such functions generate $\mathcal{H}_\tau$). This acts on both $\Hyp_\tau$ and $\mathcal{Z}_{K}(\G, \H)$ as follows: if $K_\tau g K_\tau=\sqcup g_i K_\tau$, then for $h\in G_\tau$, the corresponding endomorphism is $[h]\mapsto \sum[hg_i]$, where $[h]$ denotes the class of $h$ in either $G_\tau /K_\tau$ or the cycle $\cZ_K(\ldots,1, 1,  h, 1, 1, \ldots) \in \mathcal{Z}_{K}(\G,\H)$, respectively. 

Given a co-character $\mu$ of $\widehat\G$, there is a polynomial $H_\tau(z)$ with coefficients in $\mathcal H_\tau$, called the Hecke polynomial (this polynomial is originally defined by Langlands; we use the version of Blasius-Rogawski found in \cite[\textsection 6]{blasius-rogawski:zeta}). An explicit formula for the Hecke polynomial is given in our setting by \cite[Thm.4.1]{jetchev:unitary}.

To state that formula, let $\delta_V= \text{diag}(\varpi,1,\varpi^{-1}) \in G_{V, \tau}$ and $\delta_W= \text{diag}(\varpi,\varpi^{-1})\in G_{W, \tau}$, where matrices are written with respect to the bases chosen in Section~\ref{subsec:allowable}. Consider the Hecke operators $t_{1, 0} = \mathbf{1}_{K_\tau(\delta_V,1)K_\tau}$ and $t_{0, 1} = \mathbf{1}_{K_\tau(1,\delta_W)K_\tau}$. These act as \emph{adjacency} operators on $\Z[\Hyp_\tau]$; the former is the identity on $\Hyp_{W, \tau}$ and sends a point in $\Hyp_{V, \tau}$ to the formal sum of its neighbors, and similarly for the latter. Then one has:

 \begin{thm}\label{thm:heckepol}
 The Hecke polynomial $H_\tau(z) \in \mathcal{H}_\tau[z]$ at the place $\tau$ for the Shimura datum $(\G,X)$ is given by 
  \[H_{\tau}(z)=H^{(2)}(z)H^{(4)}(z)\]
where \[H^{(2)}(z) = z^2 - q^2(t_{0,1} - (q - 1))z + q^6\] and 
\begin{align*}H^{(4)}(z) &= z^4 \\
&+ (-t_{1,0}t_{0,1} + (q - 1)(t_{1,0} + t_{0,1}) - (q - 1)^2)z^3 \\
&+ q^2 (t_{1,0}^2 +q^2t_{0, 1}^2 -2(q-1)t_{1,0} -2q^2(q-1)t_{0,1} -q^4 -2q^3 +2q^2 -2q+1 )z^2 \\
&+ q^6(-t_{1,0}t_{0,1} + (q - 1)(t_{1,0}+t_{0,1}) - (q - 1)^2)z \\
&+ q^{12}.
\end{align*}
\end{thm}
Define the elements $C_i \in \mathcal{H}_\tau$ by $H_\tau (z) = C_0 z^6 + C_1 z^5 + \dots + C_6$ for $i = 0, 1, \dots, 6$.

\subsection{Partial Hecke operators}\label{subsec:partial_hecke}

We will make use of a formal factorization of the Hecke polynomial in a ring extension of $\mathcal{H}_\tau$. Let $R_V = \End(\Z_{(p)}[\Hyp_{V, \tau}])$, $R_W = \End(\Z_{(p)}[\Hyp_{W, \tau}]$, and $R = R_V \otimes R_W = \End(\Z_{(p)}[\Hyp_\tau])$, where $\Z_{(p)}$ denotes the localization of $\Z$ at $p$. The previously-defined actions on the buildings give an algebra map $\mathcal{H}_\tau \to R$ and a group map $H_\tau \to R^\times$.

We define predecessor and successor operators in $R_V$ and $R_W$. To keep the analogy with the case of $\GL_2$, we use the notation $\mathcal{U}$ for operators which raise the distance of a point from the origin, and $\mathcal{V}$ for operators which lower it.

Thus, given a self-dual lattice $L \neq \Lambda_V$ in $\mathcal{B}(G_{V, \tau})$, set:
\begin{itemize}
\item $\mathcal{U}_V L = \sum L'$, where the sum is taken over the self-dual lattices $L'$ of $V_\tau$  such that  $ \mathrm{dist}(L,L')=1 \text{ and } \mathrm{dist}(L', \Lambda_V) > \dist(L, \Lambda_V)$.
\item  $\mathcal{V}_V L$ is the unique self-dual lattice $L'$ of $V_\tau$ such that $\mathrm{dist}(L, L') = 1$ and $\mathrm{dist}(L', \Lambda_V) < \dist(L, \Lambda_V)$. 
\item $\mathcal{S}_V L = L + \sum L'$, where the sum is taken over the self-dual lattices $L'$ of $V_\tau$ such that $ \mathrm{dist}(L,L')=1 \text{ and } \mathrm{dist}(L', \Lambda_V) = \dist(L, \Lambda_V)$.
\end{itemize}

To complete the definition of these three operators, writing $x_V$ for the point corresponding to $\Lambda_V$, and $S_1$ for the formal sum of points of distance $1$ from the origin, and set:
\begin{itemize}
\item $\ds \mathcal{U}_V x_V = \frac{q}{q+1} S_1$
\item $\ds \mathcal{V}_V x_V = (1-q^3) x_V + \frac{1}{q+1} S_1$
\item $\ds \mathcal{S}_V x_V = q^3 x_V$
\end{itemize}
(This definition is motivated by Lemma \ref{lem:combinatorics} below.)

Define the operators $\mathcal{U}_W$, $\mathcal{V}_W$, and $\mathcal{S}_W$ analogously, replacing both instances of $q^3$ with $q$ in the above definition. We will abuse notation and consider these as elements of $R$, e.g. writing $\mathcal{U}_W$ rather than $\mathcal{U}_W \otimes 1.$  These operators do not commute with each other. They are depicted in Figure \ref{fig:operators}.
 
\begin{rem} In the definition of $R$, the localization of the ring of coefficients at $p$ is necessary only to define the operators above at the origin. The cycles occurring in the main theorem, a priori in $\mathbb{Z}_{(p)}[\cZ_K(\G, \H)]$, are in fact in $\Z[\cZ_K(\G, \H)]$. 
\end{rem}

\begin{figure}
\begin{center}
\includegraphics[width=17cm, height=10cm]{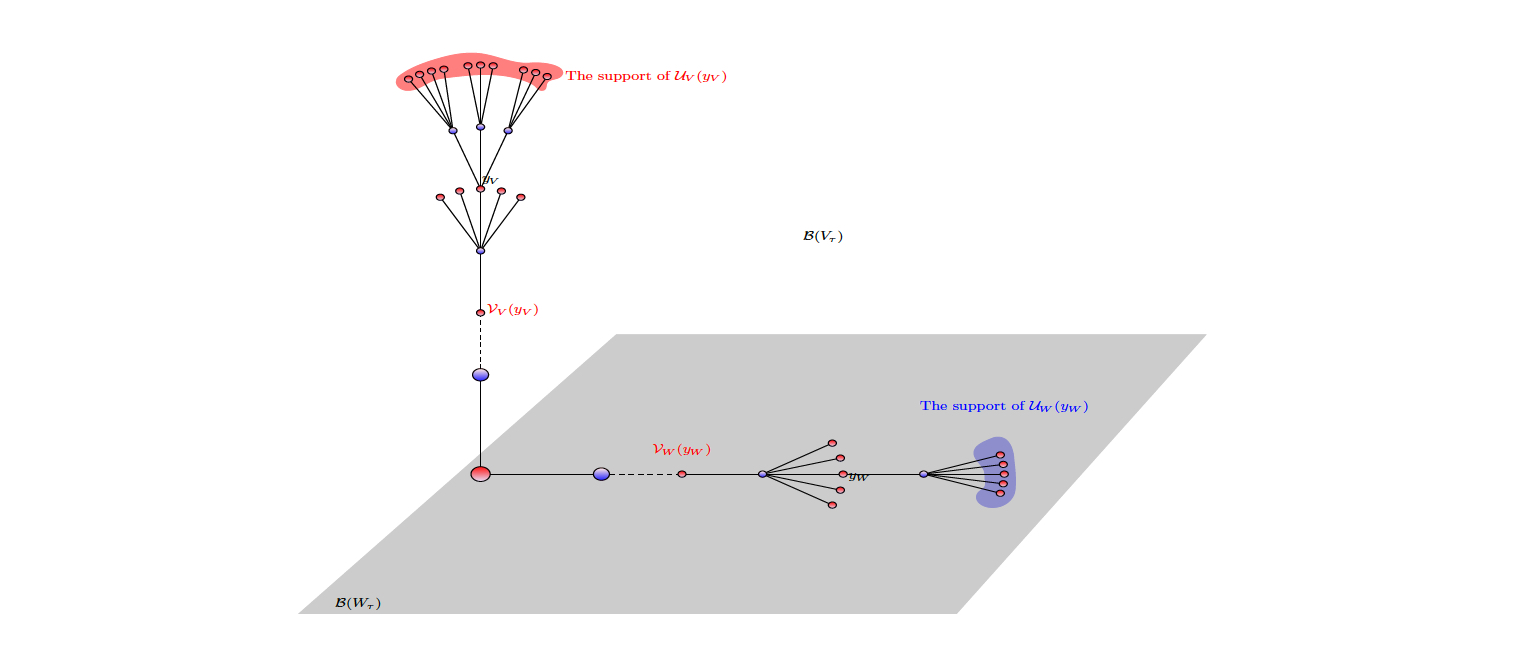}
\end{center}
\caption{Examples of elements of $R$.}
\label{fig:operators}
\end{figure}

We gather the composition relations between the various operators in the following lemma; in each case, the proof (away from the origin) is a simple counting argument.
\begin{lem}\label{lem:combinatorics}
In $R$, one has
\begin{itemize}
\item $\cV_V \cU_V = q^4$ and $\cV_W \cU_W = q^2$
\item $\cV_V + \cU_V + \cS_V = t_{1, 0} + 1$ and $\cV_W + \cU_W + \cS_W = t_{0, 1}+1$.
\item $\cV_V \cS_V = q^3 \cV_V $ and $\cV_W \cS_W = q \cV_W$.
\item $\cS_V \cU_V = q^3 \cU_V $ and $\cS_W \cU_W = q \cU_W$.
\item $(\cS_V)^2 - q^3 (\cS_V) = 0$ and $(\cS_W)^2 - q (\cS_W) = 0$.
\end{itemize}
\end{lem}

We say that an element of $\Z_{(p)}[\Hyp_{V, \tau}]$ is ``balanced'' if $\cS_V$ acts on it via $q^3$. Note that $\cU_V \cV_V = \cV_V \cU_V = q^4$ when applied to a balanced element.  We define balanced elements of $\Z_{(p)}[\Hyp_{W, \tau}]$ similarly.  An element of $\Z_{(p)}[\Hyp_\tau]$ is balanced if $\cS_V$ acts via $q^3$ and $\cS_W$ acts via $q$.

Let $R_0$ be the subring of $R$ generated by $\mathcal{H}_\tau$ and the six operators defined above, and let $\mathcal{I}$ be the quotient of $R_0$ by the relations $\cS_V = q^3$, $\cS_W = q$, and $\mathcal{U}_V\mathcal{V}_V - \mathcal{V}_V\mathcal{U}_V = \mathcal{U}_W\mathcal{V}_W - \mathcal{V}_W\mathcal{U}_W = 0$. Then $\mathcal{I}$ acts on the subgroup of $\mathbb{Z}_{(p)}[\Hyp_\tau]$ consisting of balanced elements. Moreover, $\mathcal{I}$ is a commutative ring extension of $\mathcal{H}_\tau$, so it makes sense to speak of the Hecke polynomial as an element of $\cI[z]$. Using the lemma, one calculates that it admits the following factorization there:

\[H_\tau(z)=\underbrace{(z-q^2\mathcal{U}_W)(z-q^2\mathcal{V}_W)}_{H^{(2)}(z)}\underbrace{(z-\mathcal{V}_W\mathcal{V}_V)(z-\mathcal{U}_W\mathcal{V}_V)(z-\mathcal{V}_W\mathcal{U}_V)(z-\mathcal{U}_W\mathcal{U}_V)}_{H^{(4)}(z)}\]

In particular,
\begin{lem}\label{lem:vanish}
The image of $H_\tau(z)$ in $\mathcal{I}[z]$ satisfies $H_\tau(\mathcal{V}_V\mathcal{V}_W) = 0$.
\end{lem}

%
%

\section{The main theorem}\label{sec:proof}

\subsection{Definition of the cycles}\label{subsec:cycles-def}

We now construct the sequence of special cycles $\{ \xi_n \}$ of the introduction. Recall that by Assumption \ref{assumptions} one has a cycle $\xi_0 = \xi(g_0)$ for some $g_0 \in \G(\A_f)$, defined over $L$, for which $\inv_\tau(\xi(g_0)) = (0, 0)$. By the description of the Galois action in (\ref{eq:shimurarec}), if $g \in G(\A_f)$ is such that the image of $g$ and $g_0$ in $\U(V) \times \U(W) (\A_{F, f}^{(\tau)})$ agree, then the field of definition of $\xi(g)$ will be $L[\tau^n]$, where $n$ is the local conductor of $\xi(g)$ at $\tau$. In the following, we will define $\xi_n$ by modifying $g_0$ in such a manner that only the conductor at $\tau$ changes.

Call an apartment $\cA_{V}'$ of $\cB(V_\tau)$ \emph{special} if its intersection with $\cB(W_\tau)$ is a half-line (see \cite[\textsection 3.3]{jetchev:unitary}). Let $\cA_{V}$ be the apartment defined by the 
Witt basis $\{e_+, e_0, e_-\}$ of Section \ref{subsec:allowable} and let $\cA_{V}'$ be any special apartment with the property that its intersection with 
$\cB(W_\tau)$ is the half-apartment of $\cA_{V}$ given by $\{\delta_V^n x_{V} \colon n \geq 0\}$ (see Figure~\ref{fig:apts}).

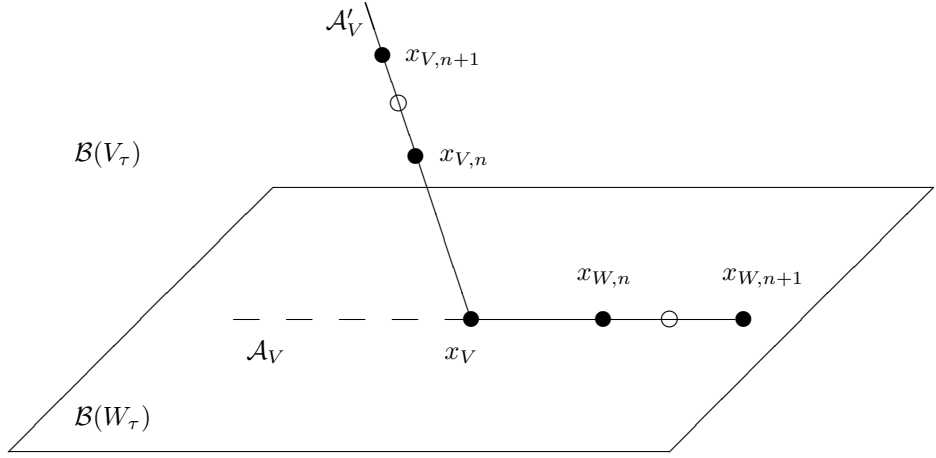
\begin{figure}
\begin{center}
\begin{picture}(300,200)

\put(0, 0){\line(250, 0){250}}
\put(0, 0){\line(1,1){100}}
\put(250, 0){\line(1,1){100}}
\put(100,100){\line(1,0){250}}
\put(175,50){\circle*{6}}
\put(90, 35){$\cA_{V}$}
\put(165, 35){$x_{V}$}
\put(250,50){\circle{6}}
\put(215, 65){$x_{W, n}$}
\put(278,50){\circle*{6}}
\put(225,50){\circle*{6}}
\put(270, 65){$x_{W, n+1}$}
\put(141.5,150){\circle*{6}}
\put(147.5,132){\circle{6}}
\put(150,148){$x_{V, n+1}$}
\put(154,112){\circle*{6}}
\put(163,110){$x_{V, n}$}
\put(120, 160){$\cA_{V}'$}

\put(175,50){\line(1,0){100}}
\multiput(175,50)(-20,0){5}{\line(-1,0){10}}
\put(25, 10){$\cB(W_\tau)$}
\put(175,50){\line(-1,3){40}}

\put(25, 110){$\cB(V_\tau)$}
\end{picture}
\end{center}
\caption{Defining $x_{n, \tau}$ via the special apartments $\cA_{V}$ and $\cA_{V}'$.}
\label{fig:apts}
\end{figure}
As explained in \cite[Lem.3.2]{jetchev:unitary}, there exists a unitary, unipotent matrix 
$$
\ds u = {\mthree {1}{\beta}{\gamma} {0}{1}{-\overline{\beta}} {0}{0}{1}} \in G_{V, \tau}
$$ 
with $\beta, \gamma \in \cO_{E_\tau}^\times$, $\beta \overline{\beta} + \gamma + \overline{\gamma} = 0$ (written in terms of the basis $\{e_+, e_0, e_-\}$) whose columns give a Witt basis determining 
$\cA_{V}'$. We then define $x_n \in \Hyp_\tau$ where  
$$
x_n := (x_{V, n}, x_{W, n}), 
$$
with $x_{V, n} = \delta_{V}^{-n} u x_{V} \in \Hyp_{V, \tau}$ and $x_{W, n} = \delta_{W}^n x_{W} \in \Hyp_{W, \tau}$.

The choice of $g_0$ furnishes us with an embedding $G_\tau/K_\tau \hookrightarrow \G(\A_f)/K$ given by $g_\tau \mapsto (g_\tau, g_0^{(\tau)})$, where $g_0^{(\tau)}$ is the image of $g_0$ in  $\U(V) \times \U(W) (\A_{F, f}^{(\tau)})$. We then have a composition
$$
\pi: \Z_{(p)}[\Hyp_\tau] = \Z_{(p)}[G_\tau/K_\tau] \hookrightarrow \Z_{(p)}[\G(\A_f)/K] \rightarrow \Z_{(p)}[\cZ_K(\G, \H)]
$$
induced by this embedding. Set $\xi_n = \pi(x_n)$.

The main properties of the cycles $\{\xi_n\}$ and the corresponding elements $x_n \in \Hyp_\tau$ are summarized in the following: 

\begin{prop}
\noindent (i) For every $n \geq 1$, $\cV_V \cV_W (x_{n+1, \tau}) = x_{n, \tau}$. 

\noindent (ii) For every $n \geq 0$, $\inv_\tau(x_n) = (n, n)$.  

\noindent (iii) For every $n \geq 1$, the cycle $\xi_n$ is defined over $L[\tau^n]$.
\end{prop}

\begin{proof}
$(i)$ By definition, $\cV_V (x_{V, n+1}, x_{W, n+1}) = (x', x_{W, n+1})$ where $x' \in \Hyp_{V, \tau}$ is the unique hyperspecial 
point of $\cB(V_\tau)$ such that $\dist(x', x_{V, n+1}) = 1$ and also $\dist(x', x_V) < \dist(x_{V, n+1}, x_V)$. 

Note that on the half-apartment of $\cA_V'$ that is outside of $\cB(W_\tau)$, the operator $\delta_V$ coincides with the operator $\cV_V$ and in the complementary half-apartment (namely, $\cA_{V}' \cap \cB(W_\tau)$), $\delta_W^{-1}$ coincides with $\cV_W$. 
Thus, 
$$
\cV_V x_{V, n+1} = \delta_V \delta_V^{-n-1} u x_{V} = \delta_V^{-n} u x_{V} = x_{V, n}.
$$ 
Similarly, one checks that $\cV_W x_{W, n+1} = x_{W, n}$ and hence that $\cV_V \cV_W(x_{n+1, \tau}) = x_{n, \tau}$. 
To prove $(ii)$, note that $\pr_{W_\tau}(x_{V, n}) = x_V$ for all $n \geq 0$ and hence, $\dist(x_{V, n}, x_V) = n$ and $\dist(x_{W, n}, x_W) = n$. Finally, $(iii)$ follows immediately from $(ii)$. 
\end{proof}

\begin{rem}
Alternatively, rather than choosing $\cA_V'$ and finding a unipotent matrix $u$, one could choose a unitary, unipotent matrix $u'$ and work with the apartment determined by $u'$. For instance, the choice 
$$
 \ds u' = {\mthree {1} {-2} {-2} {0} {1} {2} {0} {0} {1} } \in G_{V, \tau}, 
$$
gives a Witt basis whose associated apartment satisfies the correct properties. 
\end{rem}

Consider the compact open subgroups $H_n = \text{Stab}_{H_\tau}(x_n) \subset H_\tau$. 
\begin{lem}
One has $H_{n+1} \subset H_n$.
\end{lem}

\begin{proof}
Note that every element of $ H_{n+1}$ stabilizes the pairs 
$(x_{V,n+1},{\pr_{W_\tau}(x_{V,n+1})})$ and 
$(\pr_{W_\tau}(x_{V,n+1}),x_{W,n+1})$, and thus fixes (pointwise) the geodesic segment connecting the points 
$x_{V,n+1}$ and $x_{W,n+1}$. In particular, it stabilizes all of the pairs 
$(x_{V, k}, x_{W, k})$ for $k \leq n+1$, which implies the claim. 
\end{proof}

\begin{lem}
For $n \geq 1$, the group $H_n$ acts transitively on the product set $S_{n} = S_{V, n} \times S_{W, n}$ where $S_{V, n}$ (resp., $S_{W, n}$) is the set of hyperspecial points in the support of $\cU_V x_{V, n}$ (resp., $\cU_W x_{W, n}$). 
\end{lem}

\begin{proof}
Let $(x_{V, n+1}', x_{W, n+1}') \in S_n$ be any vertex. Pick a special apartment $\cA_{V}'$ of $\cB(V_\tau)$ containing 
the points $x_{V, n}, x_{W, n}, x_{V, n+1}', x_{W, n+1}'$. Such a special apartment exists by \cite[Lem.3.1]{jetchev:unitary}.
Let $\cA_{V}$ be the special apartment determined by the unipotent matrix $u \in G_{V, \tau}$ used in the definition of the $x_n$'s. By \cite[Lem.3.5]{jetchev:unitary}, there exists an element $h \in H_\tau$ moving $\cA_{V}'$ to $\cA_{V}$. Since this element $h$ necessarily fixes 
the segment connecting $x_{V, n}$ and $x_{W, n}$, it must belong to $H_n$ which proves the claim. 
\end{proof}

In particular, as the cardinality of $S_n$ is $q^6$, we obtain:
\begin{cor}\label{cor:cardinality}
For $n \geq 1$, one has $\# H_n/H_{n+1} = q^6$. 
\end{cor}

\subsection{The distribution relation}
In this section, a sum indexed by the quotient $H_n/H_{n+1}$ always means a sum over some fixed choice of coset representatives. We begin by computing traces:
\begin{lem}\label{lem:trace} For $n \geq 1$, one has
$$
\Tr_{L[\tau^{n+1}]/L[\tau^{n}]} \xi_{n+1} = \frac{1}{q^5} \pi \left(\sum_{h \in H_{n}/H_{n+1}} h \cdot x_{n+1}\right).
$$
\end{lem}
\begin{proof}
As $\tau$ is totally ramified in this extension, one may work with the local Galois group $\Gal(L_\tau[\tau^{n+1}]/L_\tau[\tau^n])$ but under Assumption 1.1A, this identifies with $\Gal(E_\tau[\tau^{n+1}]/E_\tau[\tau^n])$,  so we may assume $L = E[1]$ for the purpose of this proof.

For each $h \in H_n$, one knows that the cycle $\pi(h(x_n))$ differs from $\pi(x_n)$ by $\text{Art}_{E_\tau} \lambda$ for some $\lambda \in \mathcal{O}_n^\times$ with $\overline{\lambda}/{\lambda} = \det h$. If $h$ is replaced by an $H_{n+1}$-multiple, then $\lambda$ is replaced by an $\mathcal{O}_{n+1}^\times$-multiple. It follows that one has

$$
\pi\left(\sum_{h \in H_n/H_{n+1}} h \cdot x_{n+1}\right) = \sum_{\lambda \in \cO^\times_n / \cO^\times_{n+1}} m(\lambda) \Art_{E_\tau}(\lambda)(\xi_{n+1}),
$$ 
where $m(\lambda)$ is an unknown natural number that, a priori, depends on $n$ and $\lambda$; namely, the number of classes $h$ in $H_n/H_{n+1}$ such that $\overline{\lambda}/\lambda = \det h$. We now show that in fact $m(\lambda) = q^5$. For any $\gamma \in \mathcal{O}_n^\times \setminus \cO_{n+1}^\times$, write $\sigma = \Art_{E_\tau}^{-1}(\gamma) \in \Gal(E[\tau^{n+1}]/E[\tau^n])$, and choose $h_\gamma \in H_n$ with $\det h_\gamma = \overline{\gamma}/\gamma$. Then

\begin{align*}
\sigma \left(\sum_{\lambda \in \mathcal{O}_n^\times/\mathcal{O}_{n+1}^\times}m(\lambda){\Art_{E_\tau}(\lambda)}(\xi_{n+1})\right) &=\pi \left(\sum_{ h\in H_n/H_{n+1}}h_\gamma h \cdot x_{n+1}\right) = \\
&=\pi \left(\sum_{ h\in H_n/H_{n+1}} h \cdot x_{n+1}\right) =\\
&=\sum_{\lambda \in \mathcal{O}_n^\times/\mathcal{O}_{n+1}^\times}m(\lambda){\Art_{E_\tau}(\lambda)}(\xi_{n+1}).
\end{align*}
and so $m(\lambda) = m(\gamma\lambda)$, and, as $\gamma$ was arbitrary, this quantity does not depend on $\lambda$; we may thus unambiguously denote it by $m$. The map (of sets) $H_n/H_{n+1} \to \cO^\times_n / \cO^\times_{n+1}$ which takes $h$ to $\lambda$ such that $\det h = \overline{\lambda}/\lambda$ is then $m$-to-$1$. By Corollary \ref{cor:cardinality}, one has $m = q^6/q = q^5$ as claimed.

It follows that
\begin{align*}
\Tr_{n+1,n}(\xi_{n+1})&=\sum_{\sigma \in \Gal(E[\tau^{n+1}]/E[\tau^n])}\sigma(\xi_{n+1}) = \\
&= \sum_{\lambda\in \mathcal{O}_n^\times/\mathcal{O}_{n+1}^\times}{\text{Art}_{E_\tau}(\lambda)} (\xi_{n+1}) = \\
&=\frac{1}{q^5}\pi \left(\sum_{h\in H_n/H_{n+1}}h \cdot x_{n+1}\right).
\end{align*}
\end{proof}

We will also need a commutativity of the $H_\tau$-action on the building with the ``partial Hecke operators.'' (The genuine Hecke algebra obviously commutes with the $H_\tau$-action on the building, as the Hecke algebra is generated by adjacency operators and $H_\tau$ acts via isometries.)
\begin{lem}\label{lem:commute}
If $h \in H_1$, then $h\cV_V   = \cV_V h$ and  $h\cV_W= \cV_W h$ in $R$.
\end{lem}

\begin{proof}
For the first statement, suffices to show that the operators $h\cV_V$ and $\cV_V h$ agree on an arbitrary vertex $y_V \in \Hyp_V$. If $y_V$ is the origin, or a neighbor of the origin, then $h$ acts trivially on both $y_V$ and $\cV y_V$, so the result is clear.
Away from the origin, it follows from the definition of $\cV_V$: $h \cV_V y_V$ is a neighbor of $h y_V$, and $\dist(h\cV_Vy_V, \Lambda_V) = \dist(h\cV_Vy_V, h\Lambda_V) = \dist(\cV_vy_V, \Lambda_V) < \dist(y_V, \Lambda_V) = \dist(hy_V, h\Lambda_V) = \dist(y_V, \Lambda_V)$. The proof for $\cV_W$ and $h$ is the same.
\end{proof}

Lemmas \ref{lem:trace} and \ref{lem:commute} yield the main Theorem \ref{thm:main}. Indeed, for $n$ sufficiently large, one has:
\begin{align*}
\Tr_{L[\tau^{n+6}]/L[\tau^{n+5}]} \left (\sum_{i = 0}^6 C_i \xi_{n+6 - i} \right ) &= \frac{1}{q^5} \pi \left( \sum_{h \in H_{n+5}/H_{n+6}} h \sum_{i=0}^6 C_i x_{n+6-i}\right)\\
&= \frac{1}{q^5} \pi\left( \sum_{i=0}^6 C_i \sum_{h \in H_{n+5}/H_{n+6}} h \cdot (\cV_V\cV_W)^i x_{n+6}\right)\\
&= \frac{1}{q^5} \pi\left(\sum_{i=0}^6 C_i (\cV_V\cV_W)^i \sum_{h \in H_{n+5}/H_{n+6}} h x_{n+6}\right)
\end{align*}
As $\ds \sum_{h \in H_{n+5}/H_{n+6}} h x_{n+6}$ is balanced, the ring $R$ acts on it via the quotient $\mathcal{I}$, and this last sum is zero by Lemma \ref{lem:vanish}.

\begin{rem}\label{shorter_relation}
The same proof gives a shorter distribution relation, using the coefficients of the factor $H^{(4)}(z)$ of the Hecke polynomial $H_\tau(z)$ from Theorem~\ref{thm:heckepol}. Indeed, all that is used in the proof above is that $H_\tau(\cV_V \cV_W)$ acts as 0 on the subgroup of balanced elements of $\Z_{(p)}[\Hyp_\tau]$, and this is equally true with $H_\tau$ replaced by $H^{(4)}$, the key point being that the coefficients of $H^{(4)}$ are genuine Hecke operators and not just elements of $R$. 
\end{rem}

\section{Norm-Compatible Families}\label{sec:applications}

Let $\pi_\tau$ be a smooth admissible representation of the local group $G_\tau$ on a complex vector space such that $\dim \pi_\tau^{K_\tau} = 1$, so that the operators $C_0, C_1, \dots, C_6 \in \mathcal{H}_\tau$ act on $\pi_\tau^{K_\tau}$ by scalars, which we denote by $c_0, c_1, \dots, c_6 \in \C$, respectively; we assume that $\pi$ is \emph{algebraic} in the sense that the $c_i$ are each algebraic integers. (One expects such $\pi_\tau$ to arise from cohomological representations of the global group $\G$, but the local representation is all that is needed to build norm-compatible sequences).

For a sufficiently large number field $M$, which we may assume contains $L$, one thus has a specialization 
$$
H_\tau(z; \pi_\tau) = c_0 z^6 + \dots + c_6 \in \cO_M[z]. 
$$ 
of the Hecke polynomial $H_\tau(z)$. Let $\beta$ be a root of $H_\tau(z; \pi_\tau)$, let $\cO_\tau$ be the completion of $\cO_M$ at the place above $\tau$, and enlarge $\cO_\tau$ if needed until $\beta \in \cO_\tau$. Then $\beta^{-1}$ is a root of the polynomial $c_6 z^6 + \dots + c_0 = z^6 H_\tau(z^{-1}; \pi_\tau)$. Write  
$$
c_6z^6 + \dots + c_0 = (b_5 z^5 + \dots+ b_0)(z - \beta^{-1}) \text{ in } \cO_\tau[z]. 
$$
For $n \geq 5$, define $\widetilde y_{n} = b_5 \xi_{n} + b_4 \xi_{n-1} + \dots + b_0 \xi_{n-5} \in \cO_\tau[\cZ_K(\G, \Hbf)]$. The distribution relations imply the following:

\begin{lem}
For $n \geq 5$, one has
\begin{equation}
\Tr_{n+1, n} (\widetilde y_{n+1}) = q \beta^{-1} \widetilde y_n \in \cO_\tau[\cZ_K(\G, \Hbf)]. 
\end{equation} 
\end{lem}

\begin{proof}
The lemma is an immediate consequence of the following equality 
\begin{eqnarray*}
S &=& \Tr_{n+6, n+5} \left ( b_5 \xi_{n+6} + \dots + b_0 \xi_{n+1}  - \beta^{-1} \left (b_5 \xi_{n+5} + \dots + b_0 \xi_{n} \right ) \right ) = \\ 
&=& \Tr_{n+6, n+5} \left ( c_6 \xi_{n+6} + c_5 \xi_{n+5} + \dots + c_0 \xi_n \right ) = 0. 
\end{eqnarray*} 
\end{proof}
Now, let $\alpha = q \beta^{-1}$ and define $y_n =  \alpha^{-n} \widetilde y_n$. One then has 
\begin{equation}\label{eq:normcomp}
\Tr_{n+1, n} (y _{n+1}) = \Tr_{n+1, n} \left ( \alpha^{-n-1} \widetilde y_{n+1}\right ) = \alpha^{-n - 1} \widetilde \alpha y_n = y_n,  
\end{equation}
which gives the family of norm-compatible cycles mentioned in the introduction. A priori, these cycles appear in $\Frac(\cO_\tau)[\cZ_K(\G, \Hbf)]$ where $\Frac(\cO_\tau)$ denotes the fraction field of $\cO_\tau$; if we make in addition the ``ordinarity assumption'' that $v_\tau(\alpha) = 0$, then they are in $\cO_\tau[\cZ_K(\G, \Hbf)]$.

\section*{Acknowledgements}
We thank Christophe Cornut, Julian Rosen, and Christopher Skinner for multiple helpful discussions on this project. This research was supported by the Swiss National Science Foundation.

\bibliographystyle{amsalpha}
\bibliography{biblio-math}

\end{document}